\pdfoutput=1
\documentclass[11pt]{article}
\usepackage[a4paper,margin=28mm]{geometry}
\usepackage[T1]{fontenc}
\usepackage{lmodern,amsmath,amssymb,amsthm,mathtools,microtype}
\usepackage{graphicx,booktabs}
\usepackage[section]{placeins}
\usepackage[hidelinks]{hyperref}
\numberwithin{equation}{section}
\numberwithin{table}{section}
\newtheorem{theorem}{Theorem}[section]
\newtheorem{lemma}[theorem]{Lemma}
\newtheorem{corollary}[theorem]{Corollary}
\newtheorem{proposition}[theorem]{Proposition}
\theoremstyle{remark}
\newtheorem{remark}[theorem]{Remark}
\newtheorem{question}[theorem]{Question}
\newcommand{\Mo}{\operatorname{Mo}}
\title{Reversing the Mostar line-graph inequality\\with long pendant paths}
\hypersetup{pdftitle={Reversing the Mostar line-graph inequality with long pendant paths},
pdfauthor={Kerem Kat}}
\author{Kerem Kat\\[2pt]
\small Independent Researcher\\
\small Corresponding author: \href{mailto:keremkat@gmail.com}{\texttt{keremkat@gmail.com}}}
\date{}
\begin{document}
\maketitle
\begin{abstract}
Let $K_t$ be obtained by attaching a pendant path of length $t$ to a
fixed rooted graph $K$. We prove that $\Mo(L(K_t))-\Mo(K_t)$, the
difference between its line-graph and original Mostar indices, is
exactly affine on each parity class beyond a sharp uniform cutoff. The slope depends on
distance-level neighbor counts and is positive for every non-bipartite
core. Triangle chains give infinitely many graphs with
$\Mo(L(G))>\Mo(G)$ at every positive cyclomatic number $c$, with
maximum degree three, resolving Alex--Indulal Problem 3.3. At zero
slope, a cactus branch-mass formula decides equality; identical vertex
profiles need not give identical intercepts. Fixed cores giving
equality for all sufficiently long attachments exist exactly when $c$
is odd. Maximum degree three suffices, and at every even $c$ a
binary-tree construction gives equality on one eventual parity class.
\end{abstract}

\section{Introduction}
All graphs are finite, simple, undirected, and connected. For an edge
$uv$ of $G$, let
$n_u(uv;G)=|\{w\in V(G):d_G(w,u)<d_G(w,v)\}|$, where $d_G$ is shortest-path
distance. The \emph{Mostar index}, introduced by
Do\v{s}li\'{c} et al.~\cite{MostarOriginal}, is
\[
 \Mo(G)=\sum_{uv\in E(G)}|n_u(uv;G)-n_v(uv;G)|.
\]
Equidistant vertices enter neither count. The line graph $L(G)$ has
vertex set $E(G)$, with adjacency when two edges share an endpoint.
Write $\operatorname{cyc}(G)=|E(G)|-|V(G)|+1$.
A \emph{cactus} is a connected graph in which distinct cycles share at
most one vertex.

Sardar et al.~\cite[Conjecture 5]{Sardar} proposed
$\Mo(L(G))\leq\Mo(G)$. Alex and Indulal~\cite[Theorem 2.5]{AlexIndulal}
disproved it with infinite families at each cyclomatic number
$c\in\{1,2,3\}$. Their Problem 3.3 asks for infinitely many
counterexamples at every fixed $c\geq4$; Problems 3.1 and 3.2 ask
which non-tree graphs satisfy strict inequality and which satisfy
equality. The $[n]$circulenes of Rajabinejad and Mohammadian
Semnani~\cite[Corollary 2.3]{Circulenes} also reverse the inequality
for $n\geq4$, but their cyclomatic number $n+1$ grows with the family parameter.

For broader context, see Ali and Do\v{s}li\'{c}'s survey~\cite{AliDoslic},
Havare's bridge-attachment bounds~\cite[Theorem 3.1]{BridgeGraphs},
and Ghanbari and Alikhani's point-attachment formulas~\cite{Polymers}.
The edge Mostar index~\cite{AlexImran} counts edges nearer to the
endpoints of an edge of $G$; it differs from $\Mo(L(G))$.

\subsection{Main results}
For a \emph{rooted core} $(K,r)$, obtain $K_t$ by adjoining a pendant
path $rz_1\cdots z_t$, for integers $t\geq0$, with $K_0=K$. This
preserves cyclomatic number. Both indices grow like $t^2/2$; their
quadratic terms cancel.
\emph{Eventually affine on each parity class} means exactly
$\alpha t+\beta_0$ at even lengths and $\alpha t+\beta_1$ at odd lengths
beyond a cutoff depending on $(K,r)$, with one affine function when
$\beta_0=\beta_1$.

\begin{theorem}[Comparison law]\label{main:law}
Let $(K,r)$ be a rooted connected graph with $m$ edges. There are
integers $\alpha\geq-1$ and $\beta_0,\beta_1$, depending only on
$(K,r)$, such that
\[
 \Mo(L(K_t))-\Mo(K_t)=\alpha t+\beta_{t\bmod2}
 \qquad\text{for every integer }t\geq\max\{1,m\},
\]
and this cutoff is sharp for each $m\geq2$. The slope $\alpha$ depends
only on how many neighbors each vertex has one level nearer to $r$, at
its own level, and one level farther away.
\end{theorem}

\begin{theorem}[Reversal]\label{main:reversal}
If $K$ is not bipartite, then $\alpha\geq1$ at every root, so
$\Mo(L(K_t))>\Mo(K_t)$ for every sufficiently large $t$. In particular,
for each $c\geq1$, chains of $c$ triangles give infinitely many cacti of
cyclomatic number $c$ and maximum degree three with $\Mo(L(G))>\Mo(G)$,
and no smaller degree bound is possible.
\end{theorem}

\begin{theorem}[Equality]\label{main:equality}
There is a rooted core of cyclomatic number $c\geq1$ such that
$\Mo(L(K_t))=\Mo(K_t)$ for every sufficiently large $t$ if and only
if $c$ is odd, and it can then be chosen with maximum degree three.
For every even $c\geq2$ there is a core of maximum degree three
giving equality for every sufficiently large $t$ of one parity.
\end{theorem}

\begin{figure}[!ht]
\centering
\includegraphics[width=0.9\linewidth]{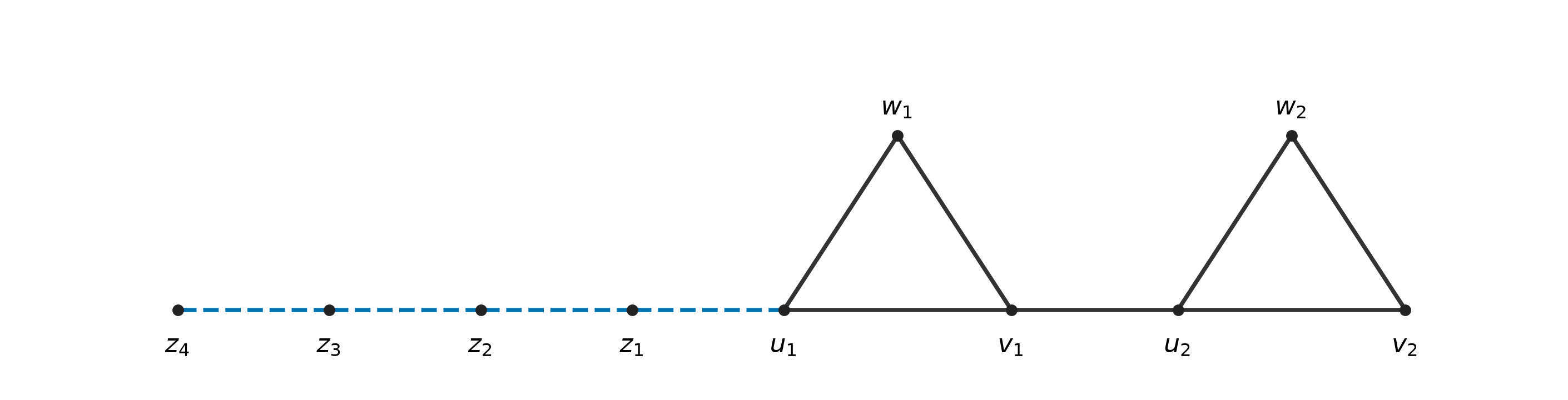}\\[2pt]
\includegraphics[width=0.76\linewidth]{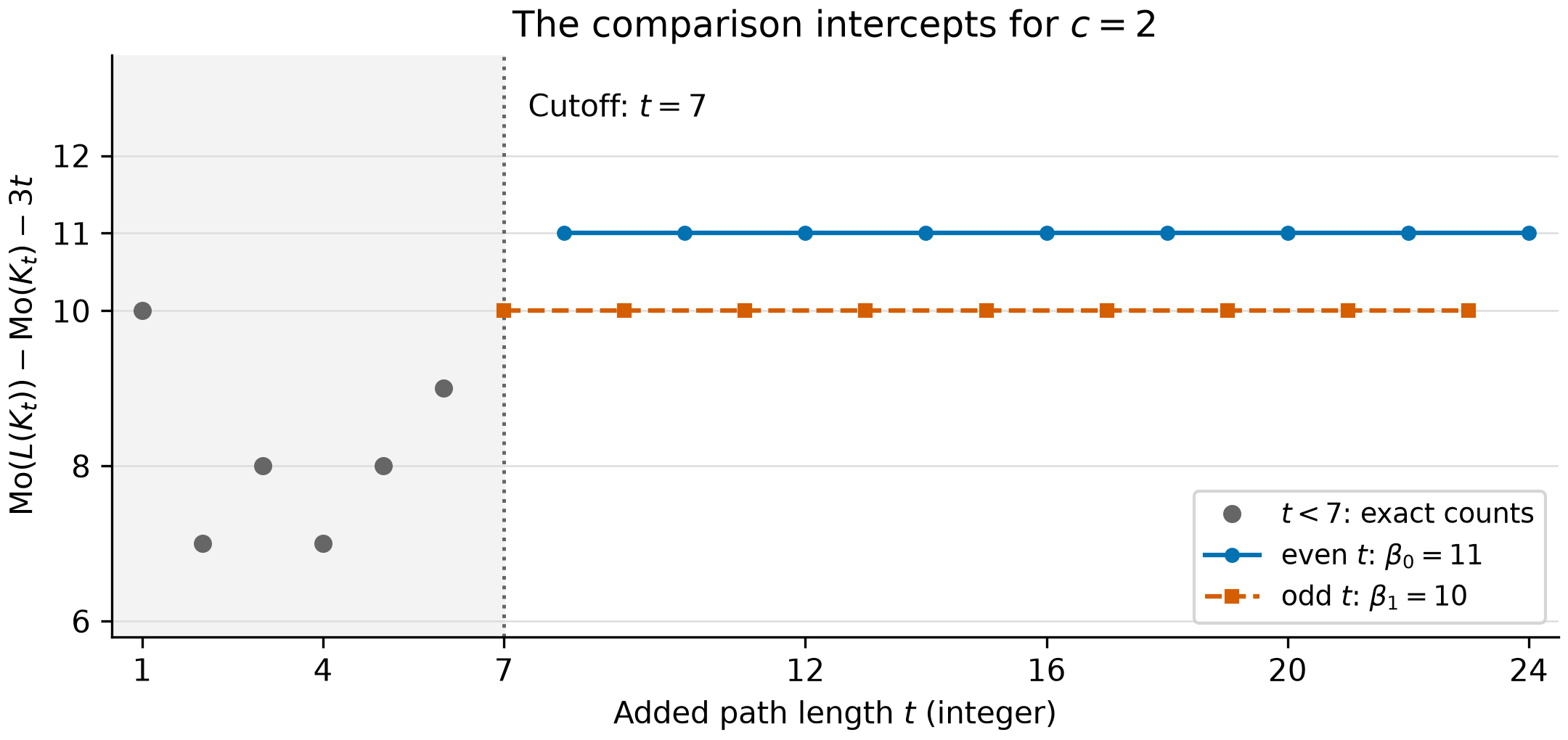}
\caption{The two-triangle chain ($c=2$, $m=7$, $\alpha=3$).
Top: the pendant path, dashed, has length $t=4$.
Bottom: exact integer-distance calculations for $1\leq t\leq24$,
after subtracting $3t$ from the comparison difference. From the cutoff
$t=7$, the values equal the intercepts $\beta_0=11$ and $\beta_1=10$,
as proved in Proposition~\ref{prop:threshold}.}
\label{fig:reversal}
\end{figure}

Theorem~\ref{main:law} is Theorem~\ref{thm:general} with
Proposition~\ref{prop:cutoff}; Theorem~\ref{main:reversal} adds
Theorem~\ref{thm:existence}, which settles Problem 3.3; and
Theorem~\ref{main:equality} combines Corollary~\ref{cor:parity} with
Theorems~\ref{thm:oddequality} and~\ref{thm:uniformequality}.
Section~\ref{sec:intercepts} computes cactus intercepts from branch
masses and shows why vertex profiles do not suffice. The appendices
contain the threshold calculation, cactus formula proof and
degree-three construction proof.

\section{The comparison law}\label{sec:path}
Fix a rooted core $(K,r)$. Put $n=|V(K)|$, $m=|E(K)|$ and $c=m-n+1$,
and call $\lambda(v)=d_K(r,v)$ the \emph{level} of $v$. Let $a_v$, $s_v$
and $h_v$ count the neighbors of $v$ at levels $\lambda(v)-1$,
$\lambda(v)+1$ and $\lambda(v)$: its \emph{predecessors},
\emph{successors} and horizontal neighbors. Thus $a_r=h_r=0$,
$s_r=\deg_K(r)$, and $a_v\geq1$ for $v\ne r$.
Table~\ref{tab:notation} collects selected recurring symbols.

\begin{table}[!ht]
\centering
\caption{Recurring notation for a rooted core $(K,r)$.}
\label{tab:notation}
\begin{tabular}{@{}ll@{}}
\toprule
Symbol & Meaning \\
\midrule
$a_v$ & Predecessor count. \\
$s_v$ & Successor count. \\
$h_v$ & Horizontal-neighbor count. \\
$S$ & Nonhorizontal edge count~\eqref{eq:sK}. \\
$B$ & Core constant~\eqref{eq:B}. \\
$\alpha$ & Comparison slope. \\
$\Pi$ & Pair-minimum sum~\eqref{eq:pairmin}. \\
$\theta_\epsilon$ & Parity term (Section~\ref{subsec:throughmass}). \\
\bottomrule
\end{tabular}
\end{table}

\subsection{Attachment and comparison}
An edge is \emph{horizontal} when its endpoints have the same level.
Orient each edge $uv$ arbitrarily and put
\[
 b_{uv}=n_u(uv;K)-n_v(uv;K),\qquad
 \sigma_{uv}=\operatorname{sgn}(\lambda(v)-\lambda(u)).
\]
Define
\begin{align}
 S=S(K,r)&=\sum_{v\ne r}a_v=\sum_v s_v,\label{eq:sK}\\
 B=B(K,r)&=\sum_{\sigma_{uv}\ne0}\sigma_{uv}b_{uv}
                  +\sum_{\sigma_{uv}=0}|b_{uv}|.\label{eq:B}
\end{align}
Here $S$ counts nonhorizontal edges, and $B$ is independent of their
orientations. For integers $n\geq1,t\geq0$, write
\begin{equation}\label{eq:F}
 F_n(t)=\sum_{j=1}^t|n-t+2j-2|.
\end{equation}
For $m>0$, Tratnik's quotient formula~\cite[Corollary 3.5]{Tratnik}
gives the next identity using $E(K)$ and the individual pendant edges
as partition blocks. This partition is admissible because every bridge forms a
singleton class of the transitive closure of the Djokovi\'{c}--Winkler relation.
The core quotient has root weight $t+1$ and all other
vertex weights one. The direct proof retains the signed contributions
needed below.

\begin{lemma}[Attachment identity]\label{lem:path}
For every integer $t\geq0$,
\begin{equation}\label{eq:attachment}
 \Mo(K_t)=F_n(t)+\sum_{uv\in E(K)}|b_{uv}+\sigma_{uv}t|.
\end{equation}
For $t\geq\max\{0,n-2\}$, this becomes
\begin{align}
 \Mo(K_t)&=F_n(t)+St+B,\label{eq:eventual}\\
 F_n(t)&=\left\lfloor\frac{t^2+(n-1)^2}{2}\right\rfloor.\label{eq:Fclosed}
\end{align}
\end{lemma}
\begin{proof}
Deleting the new edge $z_{j-1}z_j$, with $z_0=r$, leaves components
of orders $n+j-1$ and $t-j+1$, giving its summand in $F_n(t)$.
Every new vertex favors the lower-level endpoint of an old edge,
or is tied if that edge is horizontal. This gives~\eqref{eq:attachment}.
Since $|b_{uv}|\leq n-2$, its nonhorizontal terms equal
$t+\sigma_{uv}b_{uv}$ once $t\geq n-2$; for $n=1$ there are no old edges.

For $t\geq n-1$, put $j_0=\lfloor(t-n+1)/2\rfloor$.
Reversing the signs of the first $j_0$ summands gives
\[
 F_n(t)=t(n-1)+2j_0(t-n+1-j_0)
       =\left\lfloor\frac{t^2+(n-1)^2}{2}\right\rfloor.
\]
At $t=n-2\geq0$, the sum is $(n-2)(n-1)$, which equals the same
floor. This proves the endpoint as well.
\end{proof}

\begin{theorem}[General pendant-path comparison]\label{thm:general}
Let $(K,r)$ be a rooted connected graph with $n$ vertices and $m$
edges. There are integers $\beta_0,\beta_1$ depending only on $(K,r)$
such that, for every integer $t\geq\max\{1,m\}$,
\begin{equation}\label{eq:comparison}
 \Mo(L(K_t))-\Mo(K_t)=\alpha(K,r)\,t+\beta_{t\bmod2},
\end{equation}
with slope
\begin{equation}\label{eq:nonnegative}
 \alpha(K,r)=-1+\sum_{v\ne r}(a_v-1)s_v
 +\sum_{uv\ \mathrm{horizontal}}(a_u+a_v).
\end{equation}
In particular $\alpha(K,r)\geq-1$, and:
\begin{enumerate}
\item if $K$ is not bipartite, then $\alpha(K,r)\geq1$ at every root,
so $\Mo(L(K_t))>\Mo(K_t)$ for every sufficiently large $t$;
\item the slope is $-1$ if and only if $K$ is bipartite and every
vertex with at least two predecessors has no successor;
\item the slope is $0$ if and only if $K$ is bipartite, exactly one
vertex has $a_v=2$ and $s_v=1$, and every other vertex with at least
two predecessors has no successor.
\end{enumerate}
\end{theorem}
Let $H=K+rz$, where $z$ is a new leaf, and put $J=L(H)$, rooted at
$p=rz$. This \emph{line core} has $m+1$ vertices, and $L(K_t)$ is $J$
with a pendant path of length $t-1$ at $p$, for $t\geq1$.
Write $S'=S(J,p)$ and $B'=B(J,p)$.
Figure~\ref{fig:mechanism} shows the two cores and the incident-edge count.

\begin{figure}[!ht]
\centering
\includegraphics[width=\linewidth]{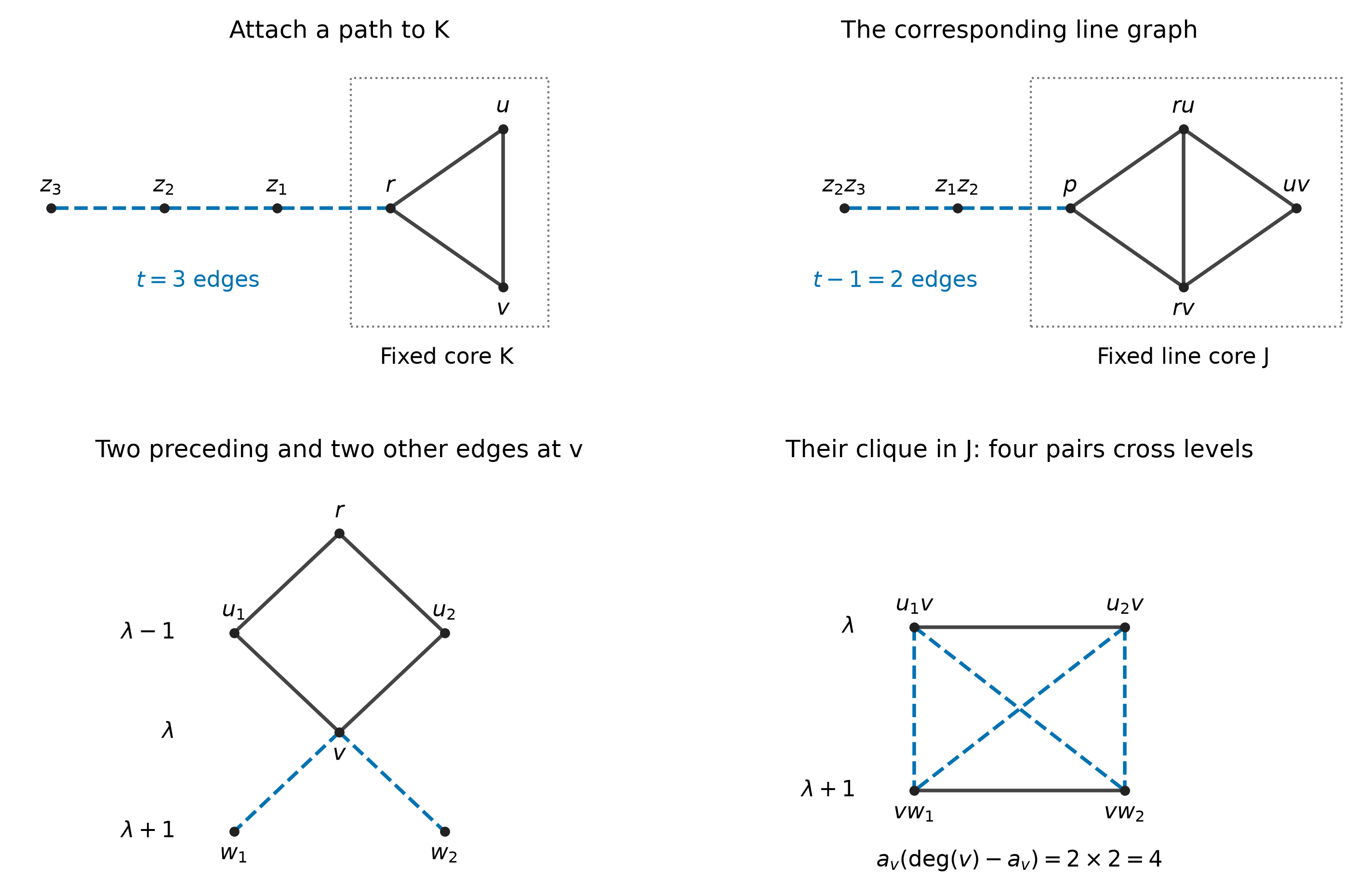}
\caption{Top: the two fixed cores, dotted, and the growing path, dashed;
its first edge becomes the line-core root $p$. Bottom: a vertex with
$a_v=2$ and two farther incident edges gives the four level-crossing
pairs counted in~\eqref{eq:sJ}; here $\lambda=\lambda(v)$.}
\label{fig:mechanism}
\end{figure}

\begin{proof}
For an edge $e=uv$ of $K$, viewed as a vertex of $J$,
\begin{equation}\label{eq:distance}
 d_J(p,e)=1+\min\{\lambda(u),\lambda(v)\}.
\end{equation}
A shortest path to the nearer endpoint supplies one inequality;
a line path to $e$ gives a walk to an endpoint and supplies the other.
Thus at $v\ne r$ the predecessor edges have level $\lambda(v)$ in $J$,
and the other incident edges have level $\lambda(v)+1$.
Each nonhorizontal pair contains one of each, while at $r$ only
pairs containing $p$ count. Distinct edges share at most one endpoint, so
\begin{equation}\label{eq:sJ}
 S'=\deg_K(r)+\sum_{v\ne r}a_v(\deg_K(v)-a_v).
\end{equation}
For $t\geq\max\{1,m\}$, Lemma~\ref{lem:path} applies to both
$(K,r)$ at length $t$ and $(J,p)$ at length $t-1$, since $m\geq n-1$.
This includes $m=0$, when both cores are single vertices.
Subtracting the two attachment formulas gives the exact identity
\begin{equation}\label{eq:generalexact}
 \Mo(L(K_t))-\Mo(K_t)
 =F_{m+1}(t-1)-F_n(t)+(S'-S)t+B'-B-S'.
\end{equation}
Let $\rho(q)\in\{0,1\}$ be the residue modulo two.
Expanding the floors in~\eqref{eq:Fclosed} gives
\[
 F_{m+1}(t-1)-F_n(t)
 =-t+\frac{1+m^2-(n-1)^2-\rho(t+m-1)+\rho(t+n-1)}{2}.
\]
The two residues depend only on the parity of $t$, so
\eqref{eq:generalexact} is~\eqref{eq:comparison} with
$\alpha(K,r)=S'-S-1$ and, for $\epsilon\in\{0,1\}$,
\begin{equation}\label{eq:beta}
 \beta_\epsilon=
 \frac{1+m^2-(n-1)^2-\rho(\epsilon+m-1)+\rho(\epsilon+n-1)}{2}
 +B'-B-S'.
\end{equation}
The slope $S'-S-1$ and both indices are integers, so
\eqref{eq:comparison} makes $\beta_0$ and $\beta_1$ integers as well.
Equations~\eqref{eq:sJ} and~\eqref{eq:sK} give
\begin{equation}\label{eq:alpha}
 \alpha(K,r)=\deg_K(r)-1+
 \sum_{v\ne r}a_v(\deg_K(v)-a_v-1).
\end{equation}
To obtain~\eqref{eq:nonnegative}, substitute $\deg_K(v)=a_v+h_v+s_v$
for $v\ne r$, $\deg_K(r)=s_r$, and $\sum_{v\ne r}a_v=\sum_v s_v$
from~\eqref{eq:sK}; this yields
\[
 \alpha(K,r)=-1+\sum_{v\ne r}(a_v-1)s_v+\sum_v a_vh_v.
\]
Each horizontal edge $uv$ contributes $a_u+a_v$ to the last sum, and
neither endpoint is the root, so this contribution is at least two.
All other summands are nonnegative integers.

There is no horizontal edge exactly when the even and odd levels form
a bipartition; conversely, in a bipartite graph the parity of the level
determines the part, so every edge joins different levels. Thus a
non-bipartite core has $\alpha\geq1$. The slope is $-1$ exactly when
there is no horizontal edge and $(a_v-1)s_v=0$ for every $v\ne r$. It is
$0$ exactly when there is no horizontal edge and the products
$(a_v-1)s_v$ sum to one, that is, precisely one of them equals one,
which forces $a_v=2$ and $s_v=1$ there, and all others are zero.
\end{proof}

\subsection{Reversal}
\begin{corollary}[Unique shortest paths]\label{cor:rooted}
If each vertex of $K$ has a unique shortest path from $r$, then
$\alpha(K,r)=2c-1$. In particular, every rooted tree has slope $-1$.
\end{corollary}
\begin{proof}
The hypothesis gives $a_v=1$ for every $v\ne r$. Thus~\eqref{eq:alpha}
and the degree sum give
$\alpha=\deg_K(r)-1+\sum_{v\ne r}(\deg_K(v)-2)=2m-2n+1=2c-1$.
\end{proof}

\begin{theorem}[Reversal at fixed cyclomatic number]
\label{thm:existence}
For each integer $c\geq1$, there are infinitely many pairwise
nonisomorphic cacti of cyclomatic
number $c$ and maximum degree three with $\Mo(L(G))>\Mo(G)$.
Specifically, a chain of $c$ triangles joined by bridges, rooted at
a degree-two vertex of an end triangle, has this property after
every path attachment of length $t\geq4c-1$.
\end{theorem}
\begin{proof}
Join disjoint triangles $u_iv_iw_iu_i$, $1\leq i\leq c$,
by bridges $v_iu_{i+1}$ and attach the path at $u_1$.
This cactus has $n=3c$, $m=4c-1$, root degree two and maximum degree
at most three. Each vertex has a unique shortest path from the root, so
Corollary~\ref{cor:rooted} gives slope $2c-1>0$. Varying $t$ therefore
gives infinitely many counterexamples, all with cyclomatic number $c$ and
maximum degree three.
The explicit threshold follows from the exact formula in
Proposition~\ref{prop:threshold}.
\end{proof}
With Theorem~\ref{thm:general}(1), this is Theorem~\ref{main:reversal}.
The degree bound is optimal: a connected graph of maximum degree
two is a path or cycle; $L(C_k)\cong C_k$, $L(P_k)\cong P_{k-1}$,
and $\Mo(P_k)=\lfloor(k-1)^2/2\rfloor$ is nondecreasing.
For $c=1$ these graphs are the unicyclic family of
\cite[Theorem 2.5, Case I]{AlexIndulal}, whose computed difference
$|V(K_t)|-4$ equals the value $t-1$ of
Proposition~\ref{prop:threshold}; the existence statement is new for
$c\geq4$. The threshold $4c-1$ is where the exact formula begins; it is
sufficient and not claimed to be minimal.

Bipartite cores show both signs. The core $K_{2,c+1}$ has cyclomatic
number $c$ and slope $-1$ at every root, since every vertex with several
predecessors is a \emph{sink}, a vertex with no successor; it gives
infinitely many graphs with
$\Mo(L(G))<\Mo(G)$ at each $c\geq1$. The cube $Q_3$ has slope $2$ at
every root, from its three vertices with $(a_v,s_v)=(2,1)$.

\subsection{Cutoff and parity}
\begin{proposition}[Sharp cutoff and preceding boundary]\label{prop:cutoff}
For a rooted connected core with $m\geq2$ edges,
formula~\eqref{eq:comparison} also holds at $t=m-1$ exactly when
$r$ is a leaf; otherwise the actual difference is two smaller.
The uniform cutoff $t\geq m$ is sharp for each $m\geq2$.
\end{proposition}
\begin{proof}
Since $m-1\geq n-2$,
the original core terms have stabilized and its bridge sum is on
the quadratic branch. The line-core bridge sum, however, satisfies
\[
 F_{m+1}(m-2)=m(m-2)
 =\left\lfloor\frac{(m-2)^2+m^2}{2}\right\rfloor-2.
\]
Orient each nonhorizontal edge $xy$ of $J$ away from $p$ and write
$b_{xy}$ for its signed imbalance. At line-path length $m-2$,
the actual contribution exceeds $m-2+b_{xy}$ only if
$b_{xy}=-(m-1)$, since $|b_{xy}|\leq m-1$; the excess is then two.
This extreme imbalance means that only $x$ is closer to $x$, and
all other vertices are closer to $y$. Thus $x$ has no neighbor
other than $y$: every other neighbor would be closer to $x$ or tied.
As $x$ is nearer to $p$, this forces $x=p$ and $\deg_J(p)=1$.
Conversely, if $p$ is a leaf, its unique edge has precisely this
imbalance. Horizontal edges contribute no correction.
Since $\deg_J(p)=\deg_K(r)$, the old line-core terms therefore add
two exactly when $r$ is a leaf. They cancel the bridge deficit in
that case; otherwise the difference is two below~\eqref{eq:comparison}.
This proves the boundary assertion. Centre-rooted stars $K_{1,m}$
then prove that the uniform cutoff is sharp for every $m\geq2$.
\end{proof}
Together with Theorem~\ref{thm:general}, this is Theorem~\ref{main:law}.

\begin{corollary}[Parity obstruction]\label{cor:parity}
The intercepts satisfy
\begin{equation}\label{eq:paritydifference}
 \beta_1-\beta_0=
 \begin{cases}
 0,&c\text{ odd},\\
 (-1)^{n-1},&c\text{ even}.
 \end{cases}
\end{equation}
Equality for every sufficiently large $t$ is equivalent to
$\alpha=\beta_0=\beta_1=0$, and requires odd $c$.
\end{corollary}
\begin{proof}
In~\eqref{eq:beta}, the two residues coincide for odd $c$.
For even $c$ they are complementary, and their half-difference
changes by $1-2\rho(n-1)=(-1)^{n-1}$ between the two parities.
Apply~\eqref{eq:comparison} on each parity class.
\end{proof}
This gives the ``only if'' direction of Theorem~\ref{main:equality}.
Gao, Xu and Do\v{s}li\'{c}~\cite[Section 4]{GaoIrregularity} give earlier
parity results for $\Mo(G)$ on cacti; Corollary~\ref{cor:parity}
concerns the line-graph difference along a pendant-path family.

\section{Branch masses and the intercept}\label{sec:intercepts}\label{subsec:throughmass}
At zero slope, the intercepts decide equality. A square with a leaf
opposite the root has $\alpha=\beta_0=\beta_1=0$, giving equality for
every $t\geq5$ (Proposition~\ref{prop:leafpath}). We compute these
constants for larger cores using branch masses.

The next lemma expresses line-graph imbalances as differences of edge
counts in the original graph, providing the intercept reduction in
Lemma~\ref{lem:massconstant}.

For an edge $g=xy$ of a connected graph $H$, let $d_H(g,a)$ be the
smaller of $d_H(x,a)$ and $d_H(y,a)$. For adjacent vertices $v,u$ define
the \emph{through-mass}
\[
 M_v(u)=1+\bigl|\{g\in E(H)\setminus\{vu\}:
                         d_H(g,u)<d_H(g,v)\}\bigr|.
\]
It counts the edges reached from $v$ through $u$, including $vu$.
Write $\nu(e,f)$ for the number of vertices of $L(H)$ nearer $e$ than $f$.

\begin{lemma}[Through-mass identity]\label{lem:throughmass}
For distinct incident edges $e=vu$ and $f=vw$ of every connected graph $H$,
\[
 \nu(e,f)-\nu(f,e)=M_v(u)-M_v(w).
\]
\end{lemma}
\begin{proof}
For distinct edges $g,e$, shortest paths between their nearest endpoints
give $d_{L(H)}(g,e)=1+\min\{d_H(g,u),d_H(g,v)\}$. Conversely, a line
path of length $j$ gives a walk of length $j-1$ between endpoints,
so equality holds. Let $g\notin\{e,f\}$. Adjacency gives
$d_H(g,u)\geq d_H(g,v)-1$. Thus $g$ is nearer $e$ than $f$ exactly when
$d_H(g,u)<d_H(g,v)$ and $d_H(g,w)\geq d_H(g,v)$.
With $\mathcal B(u)=\{g\ne vu:d_H(g,u)<d_H(g,v)\}$, this says
$\nu(e,f)=1+|\mathcal B(u)\setminus\mathcal B(w)|$; the one counts $e$,
and $f\notin\mathcal B(u)$ since $d_H(f,v)=0$. Subtract the symmetric
identity. The common part cancels, leaving the claimed difference.
\end{proof}

From now on all masses are evaluated in $H=K+rz$, and vertex sums
range over $V(K)$. Let $\mathcal E^+(v)$ be the successor edges of $K$, those joining
$v$ to level $\lambda(v)+1$, and for $f=vu$ abbreviate $M_v(f)=M_v(u)$.
In the mass computations, $p=rz$ is treated as the root's predecessor
edge, never as a successor.
The \emph{pair-minimum sum} is
\begin{equation}\label{eq:pairmin}
 \Pi=\sum_v\ \sum_{\{f,f'\}\subset \mathcal E^+(v)}\min\{M_v(f),M_v(f')\}.
\end{equation}
At maximum degree three, a vertex other than the root has at most two
successors, so it contributes at most one pair minimum. Write
$\theta_\epsilon=(\rho(\epsilon+n-1)-\rho(\epsilon+m-1))/2$
for the parity term in~\eqref{eq:beta}; it vanishes when $c$ is odd.

\subsection{One cycle}
Let $(K,r)$ be a unicyclic core. An odd cycle makes $K$ non-bipartite,
so Theorem~\ref{thm:general}(1) gives positive slope. Write an even
cycle as $x_0\cdots x_{2k-1}x_0$, with subscripts modulo $2k$ and
$x_0$ nearest $r$, and let
$\eta=d_K(r,x_0)$. Its entry antipode $x_k$ is then the only vertex
with two predecessors, so Theorem~\ref{thm:general} gives
\begin{equation}\label{eq:unicyclicslope}
 \alpha(K,r)=\deg_K(x_k)-3 .
\end{equation}
Thus degree two gives eventual strict inequality $\Mo(L(K_t))<\Mo(K_t)$,
degree at least four gives reversal, and degree three is exactly the
zero-slope case. Let $W_i$ be the number of edges of the tree attached
at $x_i$, counted in $K$, and let $s_i^{\mathrm{off}}$ count the successors of $x_i$
outside the cycle. Put $R=\sum_{0<i<k}W_i$ and
$R'=\sum_{k<i<2k}W_i$. These include arbitrary trees, not only paths.

\begin{proposition}[Unicyclic intercept formula]\label{prop:unicyclicconstant}
If the unique cycle is even and $\deg_K(x_k)=3$, then
\begin{equation}\label{eq:unicyclicconstant}
 \beta_0=\beta_1=2(n-k-\eta)+|R-R'|+\sum_{i\ne k}s_i^{\mathrm{off}}W_{i+k}-2\Pi.
\end{equation}
Consequently, equality for every $t\geq n$ is equivalent to
$2\Pi=2(n-k-\eta)+|R-R'|+\sum_{i\ne k}s_i^{\mathrm{off}}W_{i+k}$.
\end{proposition}
The proof follows the general formula in Section~\ref{subsec:cactus}.

In~\eqref{eq:unicyclicconstant}, the \emph{size} term $2(n-k-\eta)$,
the \emph{arc imbalance} $|R-R'|$ and the \emph{antipodal coupling}
$\sum_{i\ne k}s_i^{\mathrm{off}}W_{i+k}$ balance against $2\Pi$,
whose branch masses retain the sizes of entire trees.
In the notation of Section~\ref{sec:path}, the same intercept satisfies
\begin{equation}\label{eq:unicyclicbeta}
 \beta_0=\beta_1=B'-B-1,
\end{equation}
since $S=n$ and $S'=n+1$.

\begin{proposition}\label{prop:leafpath}
Let $k\geq2$, $a\geq0$, and $\ell\geq1$ be integers. Form a rooted
core $(K,v_0)$ from $C_{2k}=v_0v_1\cdots v_{2k-1}v_0$ by attaching
$a$ leaves at $v_0$ and a pendant path of length $\ell$ at $v_k$.
Then $\alpha(K,v_0)=0$ and, for every integer $t\geq2k+a+\ell$,
\begin{equation}\label{eq:leafpath}
 \Mo(L(K_t))-\Mo(K_t)=a(\ell-1-a).
\end{equation}
In particular, equality holds for all these lengths exactly when
$a=0$ or $\ell=a+1$. Moreover every integer $q$ occurs as this
eventual difference at each fixed cycle length $2k$: take
$(a,\ell)=(1,q+2)$ if $q\geq0$ and $(a,\ell)=(-q,-q)$ if $q<0$.
\end{proposition}
\begin{proof}
The antipode has degree three, so Proposition~\ref{prop:unicyclicconstant}
applies. Here $n=2k+a+\ell$, $\eta=R=R'=0$, and the coupling is $a\ell$.
Only the root has two or more successors: their masses are $a$ ones
and two copies of $k+\ell$. Thus
$\Pi=\binom a2+2a+k+\ell$. Substitution gives
$2(k+a+\ell)+a\ell-2\Pi=a(\ell-1-a)$.
\end{proof}

Inverse problems for $\Mo(G)$ itself are studied
in~\cite{InverseMostar,InverseUnicyclic,InverseCyclic,InverseBond}.

The \emph{vertex profile} is the multiset
\[
 \mathcal P(K,r)=\bigl\{\!\bigl\{
  (\lambda(v),a_v,h_v,s_v):v\in V(K)\bigr\}\!\bigr\},
\]
which fixes $n,m,S,S',\alpha$; even $(\mathcal P,B)$ does not fix the intercepts.

\begin{proposition}\label{prop:profileobstruction}
There are nine-vertex bipartite unicyclic cores $(K^{\mathrm{s}},0)$
and $(K^{\mathrm{o}},0)$ of maximum degree three and root degree two,
with zero slope, identical vertex profiles and $B=15$, but with
comparison differences $\Mo(L(K_t))-\Mo(K_t)$ equal to $2$ and $0$,
respectively, for every $t\geq9$.
\end{proposition}
\begin{proof}
Take the cycle $0123450$ and add leaves at $3$ and $1$. Add a third
leaf at $2$ for $K^{\mathrm{s}}$ and at $4$ for $K^{\mathrm{o}}$.
Exchanging vertices $2$ and $4$ matches their profile tuples; the
antipode $3$ has degree three in both. Along the outward cycle edges
$01,12,23,05,54,43$, the signed imbalances are, respectively,
$(-3,-1,1,1,-1,-3)$ and $(-1,-1,-1,-1,-1,-1)$. Each leaf edge
contributes $7$, so $B=-6+3\cdot7=15$.

In Proposition~\ref{prop:unicyclicconstant}, $n=9$, $k=3$, and $\eta=0$.
For $K^{\mathrm{s}}$, the arc imbalance is $2$, the coupling is zero,
and the only pair minima are $4$ at the root and $1$ at each of $1,2$.
Thus $\Pi=6$ and $\beta=12+2-12=2$. For $K^{\mathrm{o}}$, the arc
imbalance is zero, the coupling is $2$, and the pair minima are $5$
at the root and $1$ at each of $1,4$. Thus $\Pi=7$ and
$\beta=12+2-14=0$. Apply Theorem~\ref{thm:general}, since $m=9$.
\end{proof}
Thus the intercepts are not determined by the distance-level vertex
profile, even together with $B$; in particular, no formula depending
only on those data can parallel~\eqref{eq:nonnegative}.

\begin{figure}[!htb]
\centering
\includegraphics[width=\linewidth]{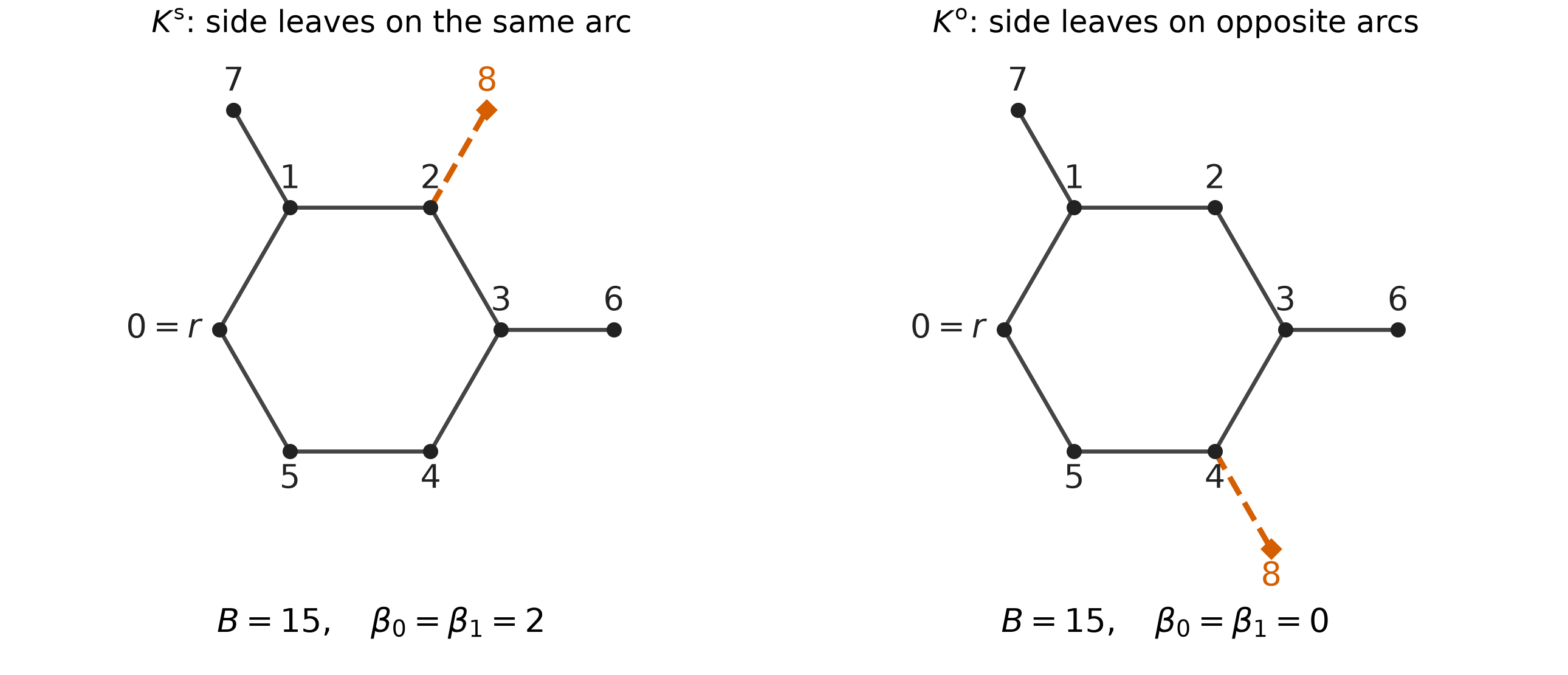}
\caption{The exact profile pair of Proposition~\ref{prop:profileobstruction}.
Moving leaf $8$ from $2$ to $4$ preserves the profile, $B$ and the sum
of arc imbalance and antipodal coupling. The smaller root-branch mass
rises by one, lowering the intercept from $2$ to $0$. The growing
root path is not shown.}
\label{fig:profileobstruction}
\end{figure}

\subsection{General cacti}\label{subsec:cactus}
The slope of every rooted cactus has a simple cycle-by-cycle form.
For related extremal Mostar problems on cacti see~\cite{HayatZhou}.

\begin{proposition}[Cactus slope]\label{prop:cactus}
For each even cycle $C$ of a rooted cactus $K$, let $w_C$ be its
vertex nearest the root and $w_C^*$ its antipode. Then
\begin{equation}\label{eq:cactusslope}
 \alpha(K,r)=2c-1+
 \sum_{C\ \mathrm{even}}\bigl(\deg_K(w_C^*)-4\bigr).
\end{equation}
If $o$ cycles are odd, then $\alpha(K,r)\geq2o-1$; if every cycle
is odd, then $\alpha(K,r)=2c-1$.
\end{proposition}
\begin{proof}
The block structure gives each cycle a unique entry nearest the root.
A shortest path enters there and then follows a shortest cycle arc:
leaving the cycle would force a return at the same cut vertex.
Each nonroot vertex receives its predecessors through a unique block
on the route from the root. It therefore has two predecessors exactly
when it is the entry antipode of an even cycle, and one otherwise.
These antipodes are distinct, since that preceding block is unique.
In~\eqref{eq:alpha}, a vertex with one predecessor contributes
$\deg_K(v)-2$, and an antipode contributes
$2(\deg_K(v)-3)=(\deg_K(v)-2)+(\deg_K(v)-4)$.
The degree sum gives the baseline $2m-2n+1=2c-1$.
Each of the $c-o$ even-cycle corrections is at least $-2$, proving
the bound.
\end{proof}

Let $K$ be a zero-slope cactus, hence bipartite by Theorem~\ref{thm:general}.
Writing $d_C=\deg_K(w_C^*)\geq2$, the cactus slope formula gives
$\sum_C(d_C-2)=1$. Hence exactly one entry antipode has degree three
and all others have degree two; call its cycle $C^*$. These conditions
follow from zero slope; degrees elsewhere are unrestricted.

\paragraph{Cycle branches.}
Write each cycle as $C=x_0x_1\cdots x_{2k-1}x_0$, with entry $x_0$
nearest the root. For $i\ne0$, let $E_i^C$ count the edges in the
component of $K-E(C)$ containing $x_i$. Let $s_i^C$ count the successor
edges of $x_i$ outside $C$, and put, with subscripts modulo $2k$,
\[
 R_C=\sum_{0<i<k}E_i^C,\quad R'_C=\sum_{k<i<2k}E_i^C,
 \quad\kappa_C=\sum_{i\ne k}s_i^CE_{i+k}^C.
\]
The excluded index ensures that $E_0^C$ is not needed in this sum.

\paragraph{Cycle depth.}
Let $D(C)=N_{\rm br}(C)+2N_{\rm cyc}(C)$, where a shortest path
from $r$ to $x_0$ uses $N_{\rm br}(C)$ bridges and $N_{\rm cyc}(C)$
cycle edges. The value is independent of choices
between equal-length arcs.

\paragraph{The special cycle.}
If $C^*$ has length $2k^*$, write
\[
 \mu^*=E_{k^*}^{C^*},\qquad
 E^*=m+1-2k^*-\sum_{i\ne0}E_i^{C^*}.
\]
These count, respectively, the edges below the special antipode and
the edges of $H$ attached at the special cycle's entry.

\begin{theorem}[Cactus intercept formula]\label{thm:cactusconstant}
For every connected rooted cactus with zero slope,
\begin{equation}\label{eq:cactusconstant}
\begin{split}
 \beta_\epsilon={}&2(c-1)n+m+\frac{3c^2-8c+3}{2}+\theta_\epsilon-2\Pi\\
 &+\sum_C\bigl(\kappa_C+2\max\{R_C,R'_C\}-2D(C)\bigr)+\mu^*+E^*.
\end{split}
\end{equation}
The masses in $\Pi$ require only edge counts: a successor bridge has
mass one plus the edges beyond it, while a successor cycle edge
$x_ix_{i\pm1}$ has mass $k+\sum_{j=1}^kE_{i\pm j}^C$.
\end{theorem}

Section~\ref{sec:scope} solves $\beta_\epsilon=0$ by balancing size and
cycle-arrangement terms against $2\Pi$. Appendix~\ref{app:cactusproof}
reduces arbitrary zero-slope cores to masses
(Lemma~\ref{lem:massconstant}), then evaluates them for cacti; only that
evaluation needs the cactus structure.

\begin{proof}[Proof of Proposition~\ref{prop:unicyclicconstant}]
The antipode is the only vertex with two predecessors, so
Theorem~\ref{thm:general} gives zero slope. In
Theorem~\ref{thm:cactusconstant}, $c=1$, $m=n$, $D(C)=\eta$,
$\mu^*=W_k$, $E^*=W_0+1$, and $\theta_\epsilon=0$.
Use $n-2k=W_0+W_k+R+R'$ and
$2\max\{R,R'\}=R+R'+|R-R'|$.
\end{proof}

\section{Equality constructions}\label{sec:scope}
For odd $c\geq3$, the bouquet gives equality cores of order $11c-20$
with unrestricted maximum degree. Binary branching achieves maximum
degree three with order $O(4^c)$, also giving one-parity equality at
even cyclomatic numbers.

\begin{lemma}[Decorated square bouquet]\label{lem:bouquet}
Let $c\geq1$, $a\geq0$, and $\ell\geq1$ be integers. Form $(K,r)$
from $c$ four-cycles sharing only $r$, with $a$ leaves at $r$ and a path of
length $\ell$ at the first cycle's antipode. We call this a
\emph{decorated square bouquet}. Then $n=3c+a+\ell+1$,
$m=4c+a+\ell$, $\alpha=0$, and
\begin{equation}\label{eq:bouquetbeta}
 \beta_\epsilon=(a+4c-4)(\ell-a+2c-3)
                 -\frac{(c-1)(17c-31)}2+\theta_\epsilon.
\end{equation}
\end{lemma}
\begin{proof}
Exactly one antipode has a successor, so the slope is zero.
In Theorem~\ref{thm:cactusconstant}, all $R_C,R'_C,D(C)$ vanish;
only the special square has nonzero coupling, $(a+2c-2)\ell$.
Also $\mu^*=\ell$ and $E^*=4c+a-3$. The successor masses at $r$ are
$a$ ones, $2c-2$ twos, and two copies of $\ell+2$. No other vertex
has two successors, giving
\[
 \Pi=\binom a2+2ac+2\binom{2c-2}{2}+8(c-1)+\ell+2.
\]
Substitution gives
$-c^2/2-2ac+4c\ell+a\ell-a^2+4c+a-4\ell-7/2+\theta_\epsilon$,
which factors as stated.
\end{proof}

\begin{theorem}[Existence of eventual equality cores]\label{thm:oddequality}
Let $c\geq1$ be an integer. There exists a connected rooted core
$(K,r)$ of cyclomatic number $c$ such that
$\Mo(L(K_t))=\Mo(K_t)$ for every sufficiently large integer $t$
if and only if $c$ is odd.
For every odd $c\geq3$, one can choose a cactus core with $11c-20$
vertices and $12c-21$ edges, giving equality for every $t\geq12c-21$.
\end{theorem}
\begin{proof}
Necessity is Corollary~\ref{cor:parity}. For $c=1$, use
Proposition~\ref{prop:leafpath} with $k=2$, $a=0$, and $\ell=1$.

For odd $c\geq3$, the parity term in Lemma~\ref{lem:bouquet} vanishes.
To make the bouquet intercept zero, choose the two factors in
\eqref{eq:bouquetbeta} to satisfy
\[
 a+4c-4=\frac{17c-31}{2},\qquad \ell-a+2c-3=c-1.
\]
Solving gives
\[
 a=\frac{9c-23}{2},\qquad \ell=\frac{7c-19}{2}.
\]
Since $c\geq3$ is odd, these are integers with $a\geq2$ and $\ell\geq1$.
The core has $n=3c+a+\ell+1=11c-20$ and
$m=4c+a+\ell=12c-21$. Theorem~\ref{thm:general} now gives equality
for every $t\geq m=12c-21$.
\end{proof}
At $c=3$ this is a 13-vertex core with two root leaves and one far
leaf. Dobrynin~\cite[Proposition 11]{Dobrynin} obtained equality of
the Szeged indices of a bouquet of even cycles and its line graph; for
the Mostar index, eventual equality at both parities forces odd $c$.

\begin{remark}[Equality on one parity class]\label{rem:singleparity}
Even cyclomatic number permits equality on one eventual parity class.
Let $v^*$ be the vertex opposite the root $r$ on a square, and join
$v^*$ to a vertex $w$ of a second, disjoint square by a path
$v^*\xi w$; attach to $\xi$ a new vertex carrying two leaves.
This core has $c=2$, $n=12$, $m=13$,
maximum degree three and zero slope. In Theorem~\ref{thm:cactusconstant},
$\Pi=11+2+3+1=17$, $\sum_CD(C)=6$, $\mu^*=9$, $E^*=1$,
and every $\kappa_C,R_C,R'_C$ vanishes. Thus
$\beta_\epsilon=1/2+\theta_\epsilon=1-\epsilon$:
equality holds for odd $t\geq13$ and the difference is one for
even $t\geq14$. Theorem~\ref{thm:uniformequality} gives such
degree-three cores for every even $c$.
\end{remark}

\subsection{Equality with maximum degree three}
Binary branching keeps degrees at most three while accumulating pair
minima to offset the size term $2(c-1)n$ in~\eqref{eq:cactusconstant}.

\begin{theorem}[Uniform degree-three equality]
\label{thm:uniformequality}
For every integer $c\geq1$ there is a rooted cactus $(K,r)$ with
cyclomatic number $c$, maximum degree three and root degree two such
that, writing $n=|V(K)|$ and $m=|E(K)|$, for every integer $t\geq m$,
\[
 \Mo(L(K_t))-\Mo(K_t)=
 \begin{cases}
 0,&c\text{ odd},\\
 \bigl(1+(-1)^{n+t}\bigr)/2,&c\text{ even}.
 \end{cases}
\]
The cores can be chosen with order $O(4^c)$. Thus the criterion of odd
cyclomatic number for equality at all sufficiently large lengths remains
exact with maximum degree three. For even $c$, equality holds whenever
the attachment is sufficiently long and the resulting graph has odd order.
\end{theorem}
With Theorem~\ref{thm:oddequality} and Corollary~\ref{cor:parity},
this completes Theorem~\ref{main:equality}.

\begin{figure}[!ht]
\centering
\includegraphics[width=\linewidth]{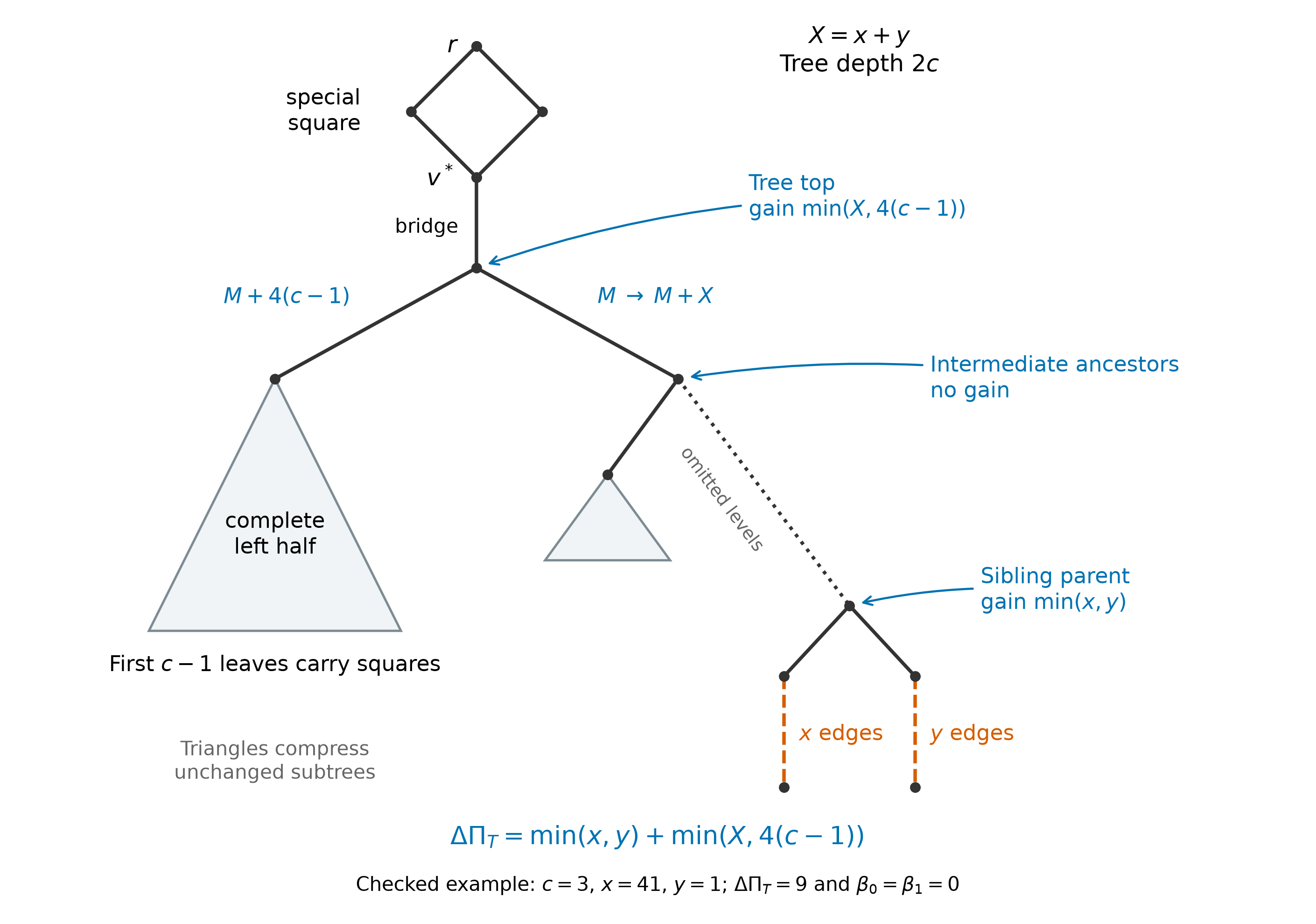}
\caption{The construction of Theorem~\ref{thm:uniformequality} for
$c\geq2$, with $v^*$ the antipode of $r$ on the root square,
$X=x+y$, and $M=2^{2c}-1$ the initial right-branch mass at
the tree top. Triangles compress unchanged subtrees and the dotted
segment suppresses intermediate levels; orange paths are not to scale.
Blue labels give the local changes of the pair-minimum sum, whose total
is $\Delta\Pi_T$.}
\label{fig:uniformequality}
\end{figure}

\paragraph{Idea of the proof.}
Hang a complete binary tree of depth $2c$ from the antipode of a square
at the root, and put $c-1$ further squares on its leftmost leaves.
Theorem~\ref{thm:cactusconstant} gives this core a negative parity-independent
intercept $\beta_\epsilon-\theta_\epsilon$ of
order $-2\cdot4^c$. Extending the two rightmost sibling leaves by paths
of lengths $x$ and $y$ raises this quantity by $2(c-1)(x+y)$ and lowers
it by twice the increase $\Delta\Pi_T=\min(x,y)+\min(x+y,4(c-1))$ of the
pair-minimum sum. Only the siblings' parent and the tree top contribute
to this increase; the intermediate ancestors have zero gain
(Figure~\ref{fig:uniformequality}). Once the second minimum is saturated,
one length fixes a residue and the other the size, so the target is met
exactly. Appendix~\ref{app:degreethree} gives the details.

\section{Concluding remarks}\label{sec:conclusion}
An odd cycle in the core is enough to reverse the Mostar line-graph
inequality after a sufficiently long path is attached at any vertex
(Theorem~\ref{main:reversal}). Triangle chains realize this at every
fixed positive cyclomatic number with maximum degree three, settling
Alex--Indulal Problem 3.3.

Equality at every sufficiently large attachment length is possible
exactly at odd cyclomatic numbers (Theorem~\ref{main:equality}). At even
cyclomatic numbers the parity term of Corollary~\ref{cor:parity}
permits equality on only one parity class. Maximum degree three
suffices in both cases.

The comparison law (Theorem~\ref{main:law}) determines an infinite
attachment family through one finite core calculation. It also enables
prescribed differences: decorations of any fixed even cycle realize
every integer eventual value (Proposition~\ref{prop:leafpath}). These
results concern fixed cores carrying growing paths. Within such
families they answer Problems 3.1 and 3.2 in part; classification of
all equality graphs remains open.

Two natural questions arise.

\begin{question}\label{q:beyondcacti}
Lemma~\ref{lem:massconstant} reduces every zero-slope intercept to
through-masses, with no cactus hypothesis. Which branch and overlap
patterns make it vanish beyond cacti?
\end{question}

\begin{question}\label{q:polynomial}
Can the $O(4^c)$ order in Theorem~\ref{thm:uniformequality} be replaced
by a polynomial bound?
\end{question}

The first asks for a structural description of equality, the second for
more efficient realizations.

\appendix
\section{The explicit threshold for triangle chains}\label{app:threshold}
\begin{proposition}\label{prop:threshold}
Let $(K,r)$ be the chain of $c\geq1$ triangles in
Theorem~\ref{thm:existence}. For every integer $t\geq4c-1$,
\[
 \Mo(L(K_t))-\Mo(K_t)
 =(2c-1)t+\frac{17c^2-28c+9}{2}+\theta_{t\bmod2}>0.
\]
Here $\theta_\epsilon$ is the parity term defined in
Section~\ref{subsec:throughmass}.
\end{proposition}
\begin{proof}
Keep the notation of the proof of Theorem~\ref{thm:existence}, so
$n=3c$ and $m=4c-1$. Counting nonhorizontal edges gives $S=3c-1$;
\eqref{eq:sJ} gives $S'=5c-1$. The outward imbalances on
$u_iv_i,u_iw_i,v_iu_{i+1}$ are $6i-3c-3,3i-3,6i-3c$,
respectively; the horizontal edge $v_iw_i$ contributes $3(c-i)$.
Summing gives $B=3c(c-1)$.

For the line core put $A_i=u_iv_i$, $U_i=u_iw_i$, $Z_i=v_iw_i$,
$Y_i=v_iu_{i+1}$ and $Y_0=p$. The vertices to the left of triangle
$i$, including $Y_{i-1}$, form a group of size $\phi_i=4i-3$;
those to its right, including $Y_i$, form a group of size $\psi_i=4(c-i)$.
Each group attaches through its named vertex, giving
\[
\begin{array}{c|rrrrrrr}
\text{edge}&A_iU_i&A_iZ_i&U_iZ_i&Y_{i-1}A_i&Y_{i-1}U_i&A_iY_i&Z_iY_i\\
\hline
B'\text{ term}&\psi_i&\phi_i&\phi_i-\psi_i&\phi_i-\psi_i-2&\phi_i-2&\phi_i-\psi_i+2&\psi_i-2
\end{array}
\]
By Lemma~\ref{lem:throughmass}, each signed difference for $vu,vw$ is
$M_v(u)-M_v(w)$, as in $A_iU_i:(\psi_i+2)-2=\psi_i$.
The first and last columns are horizontal and use absolute differences;
all others are oriented away from $p$. The last two occur only for
$i<c$, when $\psi_i\geq4$. These edges partition $E(J)$, so
$B'=\sum_{i<c}(5\phi_i-\psi_i-4)+4\phi_c-4=8c^2-11c+3$.
Substituting $S,S',B,B'$ into~\eqref{eq:beta}, Theorem~\ref{thm:general}
gives the stated formula for $t\geq m=4c-1$. It is positive: for $c=1$
it equals $t-1$, and for $c\geq2$ the constant is at least
$(17c^2-28c+8)/2>0$.
\end{proof}

\section{Proof of the cactus intercept formula}\label{app:cactusproof}
Evaluate all masses in $H=K+rz$ and all vertex sums over $V(K)$.
Let $\mathcal E^-(v)$ and $\mathcal E^+(v)$ be the predecessor and successor
edges at $v$, taking $\mathcal E^-(r)=\{p\}$; successor edges are edges of $K$.
For $e=vu$ abbreviate $M_v(e)=M_v(u)$.
If $K$ is bipartite, the line edges from $\mathcal E^-(v)\times \mathcal E^+(v)$
are oriented from predecessor to successor, and pairs within either
set are horizontal. Lemma~\ref{lem:throughmass} gives
\begin{equation}\label{eq:massB}
\begin{split}
 B'=\sum_v\biggl(&\sum_{e\in \mathcal E^-(v),\ f\in \mathcal E^+(v)}(M_v(e)-M_v(f))\\
 &+\sum_{\{e,e'\}\subset \mathcal E^-(v)}|M_v(e)-M_v(e')|
  +\sum_{\{f,f'\}\subset \mathcal E^+(v)}|M_v(f)-M_v(f')|\biggr).
\end{split}
\end{equation}
The identity $|x-y|=x+y-2\min\{x,y\}$ converts the last sum to
branch masses and the pair minima in~\eqref{eq:pairmin}.

Suppose now that the slope is zero. Theorem~\ref{thm:general} then gives
exactly one vertex $v^*$ with two predecessors and one successor $f^*$;
every other vertex with multiple predecessors is a sink. For a successor
edge $f=vu$, put $N_v(f)=n_u(f;K)$ and $\delta(f)=M_v(f)-N_v(f)$.
Define the overlap
$\Omega_v=\sum_{u\sim_H v}M_v(u)-(m+1)\geq0$; every edge of $H$ is
reached through at least one neighbor of $v$.

\begin{lemma}[Mass reduction at zero slope]\label{lem:massconstant}
For every rooted core with zero slope and every
$\epsilon\in\{0,1\}$,
\begin{equation}\label{eq:massconstant}
\begin{split}
 \beta_\epsilon={}&2(c-1)n+m+\frac{3c^2-8c+3}{2}+\theta_\epsilon-2\Pi\\
 &+\sum_{a_v\leq1}s_v\Omega_v+\Omega_{v^*}-M_{v^*}(f^*)
       -2\sum_v\sum_{f\in \mathcal E^+(v)}\delta(f)\\
 &+\sum_{a_v\geq2}\ \sum_{\{e,e'\}\subset \mathcal E^-(v)}|M_v(e)-M_v(e')|.
\end{split}
\end{equation}
\end{lemma}
\begin{proof}
The core is bipartite by Theorem~\ref{thm:general}, so~\eqref{eq:massB}
holds. Write $B'_v$ for the summand at $v$ in~\eqref{eq:massB} and
$B_v=\sum_{f\in \mathcal E^+(v)}(n-2N_v(f))$, so $B=\sum_vB_v$.
For each vertex $v$, put
$T_v=\sum_{f\in \mathcal E^+(v)}M_v(f)$ and
$\Pi_v=\sum_{\{f,f'\}\subset \mathcal E^+(v)}\min\{M_v(f),M_v(f')\}$.
If $a_v\leq1$, there is one predecessor edge $e$ in $H$, including at $r$.
Since $M_v(e)=m+1+\Omega_v-T_v$, the pair-sum identity gives
\begin{align*}
 B'_v&=s_vM_v(e)-T_v+(s_v-1)T_v-2\Pi_v
       =s_v(m+1+\Omega_v)-2T_v-2\Pi_v,\\
 B'_v-B_v&=s_v(c+\Omega_v)-2\Pi_v
                       -2\sum_{f\in \mathcal E^+(v)}\delta(f).
\end{align*}
At $v^*$, write $e_1,e_2$ for the predecessor edges. Their masses
sum to $m+1+\Omega_{v^*}-M_{v^*}(f^*)$.
Thus, writing $\Delta_{v^*}=|M_{v^*}(e_1)-M_{v^*}(e_2)|$,
\begin{align*}
 B'_{v^*}-B_{v^*}
 &=m+1+\Omega_{v^*}-3M_{v^*}(f^*)+\Delta_{v^*}-n+2N_{v^*}(f^*)\\
 &=c+\Omega_{v^*}-M_{v^*}(f^*)
                         -2\delta(f^*)+\Delta_{v^*}.
\end{align*}
Vertices with no successor contribute only predecessor-pair differences.
The successor counts for $a_v\leq1$ sum to $m-1$, and the special
vertex supplies the remaining one. Also $\sum_v\Pi_v=\Pi$, with
$\Pi_v=0$ at the special vertex and at sinks. Let $\mathcal R$ denote
the last two lines of~\eqref{eq:massconstant}. Since $S=m$, $S'=m+1$
and $m=n+c-1$, summing and applying~\eqref{eq:beta} gives
\begin{align*}
 B'-B&=cm-2\Pi+\mathcal R,\\
 \beta_\epsilon&=\frac{1+c(m+n-1)}2-(m+1)+cm
                   +\theta_\epsilon-2\Pi+\mathcal R,\\
 \frac{1+c(m+n-1)}2-(m+1)+cm
   &=(2c-1)n+\frac{3c^2-6c+1}{2}\\
   &=2(c-1)n+m+\frac{3c^2-8c+3}{2}.
\end{align*}
This gives~\eqref{eq:massconstant}.
\end{proof}

\begin{proof}[Proof of Theorem~\ref{thm:cactusconstant}]
In Lemma~\ref{lem:massconstant}, overlaps and predecessor-mass
differences give the arrangement terms; mass defects give the depths.

At a bridge, the reached edges are
exactly its branch edges. For a cycle edge $x_ix_{i\pm1}$ viewed from
$x_i$, on a cycle of length $2k$, the \emph{far half} consists of the
$k$ vertices $x_{i\pm1},\ldots,x_{i\pm k}$, those nearer to
$x_{i\pm1}$ than to $x_i$, and the $k$ edges of the cycle path
$x_ix_{i\pm1}\cdots x_{i\pm k}$. The reached edges are these $k$
edges and all edges attached to those $k$ vertices. This proves
the stated masses. Two reached sets overlap only for neighbors on the
same cycle, in the edges attached at the antipode. Therefore $\Omega_{x_i}$
is the sum of these antipodal edge counts over cycles through $x_i$.
Since $\deg_K(v^*)=3$ and two of its edges lie on $C^*$, the third is a
bridge, so $C^*$ is the only cycle through $v^*$; its antipode there is
the entry $x_0$. Hence $\Omega_{v^*}=E^*$. A nonantipodal cycle vertex has
$s_i^C+1$ successors, or $s_0^C+2$ at the entry, whence
\[
 \sum_{a_v\leq1}s_v\Omega_v
  =\sum_C(\kappa_C+R_C+R'_C+2E_k^C).
\]
Here $E_k^C=0$ except on $C^*$: by the slope classification every
other entry antipode has degree two, so nothing is attached there.
On $C^*$ it is $\mu^*$, which is also $M_{v^*}(f^*)$, the mass of the
single successor bridge at $v^*$. The two predecessor masses at a cycle's antipode
differ by $R_C-R'_C$.

To evaluate the defects, first let $f=vu$ be a successor bridge and
$T$ the component of $K-f$ containing $u$. Then
$M_v(f)=1+|E(T)|$ and $N_v(f)=|V(T)|$, so $\delta(f)$ is the number
of cycles in $T$. For a successor edge on an even cycle $C'$ of length
$2k$, let $T_1,\ldots,T_k$ be the components of $K-E(C')$ at the
$k$ vertices of its far half. This half excludes the entry of $C'$,
so these components contain neither $r$ nor the added leaf $z$. The edge mass is
$k+\sum_i|E(T_i)|$ and the nearer-vertex count is $\sum_i|V(T_i)|$.
Thus $\delta(f)=\sum_i(|E(T_i)|-|V(T_i)|+1)$ counts their cycles;
the cycle $C'$ itself is not counted.

Now fix a cycle $C$ and count the defects in which it occurs.
Only blocks on the route from $r$ to its entry contribute, with one
contribution from each bridge. Suppose an earlier cycle
$C'=x_0\cdots x_{2k-1}x_0$ is entered at $x_0$ and exited at $x_j$.
Reverse its orientation if needed so that $0\leq j\leq k$.
On the first arc, the far half of $x_ix_{i+1}$ contains $x_j$
exactly for $0\leq i<j$. On the other arc, the far half of
$x_{2k-i}x_{2k-i-1}$ contains $x_j$ exactly for $k-j\leq i<k$,
with subscripts modulo $2k$. There are $j$ contributions on each
arc, including zero at $j=0$ and all $2k$ edges at $j=k$.
Each traversed cycle arc therefore contributes twice its length.
Interchanging the two finite counts gives
\begin{equation}\label{eq:cycledepthcount}
 \sum_v\sum_{f\in \mathcal E^+(v)}\delta(f)
 =\sum_C\bigl(N_{\rm br}(C)+2N_{\rm cyc}(C)\bigr)=\sum_CD(C).
\end{equation}
Substitute these evaluations into
\eqref{eq:massconstant} and use
$R_C+R'_C+|R_C-R'_C|=2\max\{R_C,R'_C\}$ to obtain the formula.
\end{proof}

\section{The degree-three construction}\label{app:degreethree}
Let $\mathcal F(c,d_T)$ consist of a square through the root $r$ with
antipode $v^*$, a bridge from $v^*$ to the top of a complete binary
tree of depth $d_T$, and a
bare square at each of the tree's first $c-1$ leaves, formed by
identifying that leaf with one vertex of a fresh square.
Here $d_T\geq1$ and $c-1\leq2^{d_T}$. Give the tree heap indices
$1,\ldots,2^{d_T+1}-1$, with children $2j,2j+1$ at $j$; the first
leaves are $2^{d_T},\ldots,2^{d_T}+c-2$. The root $r$ has degree two and all
degrees are at most three. The special antipode is $v^*$.
Let $\Gamma$ sum, over tree branching vertices, the smaller number
of attached squares in their two child subtrees.

\begin{proposition}[Binary branching]\label{prop:binaryfamily}
For all integers $c\geq1$ and $d_T\geq1$ with $c-1\leq2^{d_T}$, the
intercepts of $\mathcal F(c,d_T)$ are
\begin{equation}\label{eq:binaryfamily}
 \beta_\epsilon=2^{d_T+1}(2c-1-d_T)+\frac{15c^2-48c+29}{2}
                      -2(c-1)d_T-8\Gamma+\theta_\epsilon.
\end{equation}
Adding a pendant vertex at a tree vertex outside the cycles changes
$\beta_\epsilon-\theta_\epsilon$ by $2(c-1)-2\Delta\Pi_T$, where $\Delta\Pi_T$ is the resulting
change of the pair-minimum sum inside the tree. More generally, a
sequence of such additions and their inverse deletions, preserving
the cycles, their incidences and entry depths, changes it by
$2(c-1)\Delta n-2\Delta\Pi_T$.
\end{proposition}
\begin{proof}
Here $n=2^{d_T+1}+3c$. Every cycle has $\kappa_C=R_C=R'_C=0$;
$D(C^*)=0$, and $D(C)=5+d_T$ for each pendant square. Also
$\mu^*=m-4$ and $E^*=1$. The two root successor edges both have mass
$m-2$, and each pendant square contributes pair minimum $2$ at its
entry. If a tree vertex's child subtrees have height $g$, its pair
minimum is $2^{g+1}-1$ plus four times the smaller number of squares
below them. Summing gives
\[
 \Pi=(m-2)+2(c-1)+(d_T-1)2^{d_T}+1+4\Gamma.
\]
Theorem~\ref{thm:cactusconstant} now gives~\eqref{eq:binaryfamily}.
An added tree vertex raises $n,m,\mu^*$ and the root pair minimum
by one; the arrangement and cycle depths stay fixed. The net change
outside the tree's pair minima is $2(c-1)$, proving the edit rule.
The sequence rule follows by summing these changes.
\end{proof}

\begin{proof}[Proof of Theorem~\ref{thm:uniformequality}]
For $c=1$, use Proposition~\ref{prop:leafpath} with $k=2,a=0,\ell=1$.
For $c\geq2$, start from $\mathcal F(c,2c)$.
This choice of depth makes $2c-1-d_T=-1$ in~\eqref{eq:binaryfamily},
giving a negative leading term in the intercept. The path extensions
below will raise its parity-independent part to the target value.
All $c-1$ pendant squares lie in the left half of the tree, since
$c-1\leq2^{2c-1}$. Extend the rightmost two sibling leaves by paths of
lengths $x,y\geq0$, and put $X=x+y$. These leaves lie on no cycle,
so the extensions preserve zero slope and both degree requirements.
By Proposition~\ref{prop:binaryfamily}, the original intercept is
$b_c+\theta_\epsilon$, where
\[
 b_c=-2\cdot4^c+\frac{7c^2-40c+29}{2}-8\Gamma.
\]

At the siblings' parent, the pair minimum rises by $\min(x,y)$.
Figure~\ref{fig:uniformequality} marks the affected branching vertices.
At every intermediate ancestor in the right half, the child masses
were tied; the edited child becomes heavier, leaving the minimum fixed.
At the tree top, the left child was heavier by $4(c-1)$, so the minimum
rises by $\min(X,4(c-1))$. All other tree pairs are unchanged. The edit
rule, with $\theta_\epsilon$ evaluated at the new order, therefore gives
\begin{equation}\label{eq:siblingintercept}
 \beta_\epsilon=b_c+2(c-1)X-2\min(x,y)-2\min(X,4(c-1))+\theta_\epsilon.
\end{equation}
For the $c=3$ example in Figure~\ref{fig:uniformequality}, $x=41$ and $y=1$ give
\[
 \Gamma=1,\qquad b_3=-150,\qquad X=42,\qquad \theta_\epsilon=0.
\]
The two pair-minimum gains are $1$ and $8$, the latter saturated at
$4(c-1)$, so
\[
 \beta_\epsilon=-150+4(42)-2(1)-2(8)=0.
\]

Put $q=0$ for odd $c$ and $q=1$ for even $c$. To obtain the stated
difference, we seek a parity-independent intercept of $q/2$.
In the regime $x\geq y$ and $X\geq4(c-1)$, the two minima in
\eqref{eq:siblingintercept} are $y$ and $4(c-1)$, so this target
is equivalent to
\[
 (c-1)X-y=4(c-1)+\frac{q-2b_c}{4}.
\]
Denote the right-hand side by $A$. The formula for $b_c$ gives
$2b_c\equiv3c^2+1\equiv q\pmod4$, so
\[
 A=4(c-1)+\frac{q-2b_c}{4}
   =4^c+\frac{-7c^2+56c-45+q}{4}+4\Gamma
\]
is an integer. Moreover $A\geq4(c-1)^2$: for $c=2$ we have $\Gamma=0$
and $A=26$; for $c\geq3$, use $A\geq4^c-2c^2$ and
$4^c\geq6c^2$, the latter following by induction from $c=3$.
Choose $y$ as the least nonnegative residue of $-A$ modulo $c-1$, and set
\[
 X=(A+y)/(c-1),\qquad x=X-y.
\]
Then $0\leq y\leq c-2$, $X\geq4(c-1)$, $x\geq y$, and $(c-1)X-y=A$.
Thus the parity-independent part of~\eqref{eq:siblingintercept} is
$b_c+2A-8(c-1)=q/2$. For odd $c$, $q=\theta_\epsilon=0$; for even $c$,
$\theta_\epsilon=(-1)^{n+\epsilon}/2$. Theorem~\ref{thm:general}
gives the stated difference for every $t\geq m$.
At each tree level the smaller child counts sum to at most $(c-1)/2$,
so $\Gamma\leq c(c-1)$. Consequently $A=4^c+O(c^2)$ and
$n=2\cdot4^c+3c+(A+y)/(c-1)=O(4^c)$.
Attaching the root path preserves maximum degree three; increasing
its length within the equality parity classes gives infinitely many
pairwise nonisomorphic graphs.
\end{proof}

\section*{Declarations}
\paragraph{Computational verification.}
The ancillary material contains verification code, exact
records and replay instructions. They check the comparison law, its
slope, parity and cutoff at every root of every connected graph with at
most seven vertices, and test the cactus formula and finite instances
of the constructions using integer distances, including the uniform
degree-three family for $2\leq c\leq5$.
The proofs do not depend on these finite checks.

\paragraph{Use of artificial intelligence.}
During preparation of this work, the author used Fable~5.1 (Claude Code)
and Astra~6.0 (OpenAI Codex) to assist with exploration, manuscript revision
and verification code. After using these tools, the author reviewed
and edited the content as needed and takes full responsibility for the
content of the publication.

\paragraph{Conflicts of interest.}
The author declares no conflicts of interest.

\begingroup
\small
\bibliographystyle{plain}
\bibliography{references}
\endgroup
\end{document}